\documentclass[12pt, a4paper]{amsart}
\usepackage{amscd,amsthm,amsfonts,amssymb,amsmath,euscript}
\usepackage[dvips]{graphicx}
\usepackage[dvips]{graphics}
\usepackage[matrix,arrow]{xy}
\usepackage{longtable}
\usepackage{color}
\usepackage{url}

\newcounter{statements}

\newtheorem{theorem}[statements]{Theorem}
\newtheorem{conjecture}[statements]{Conjecture}
\newtheorem{problem}[statements]{Problem}
\newtheorem{question}[statements]{Question}

\newtheorem{lemma}[statements]{Lemma}

\newtheorem{corollary}[statements]{Corollary}
\newtheorem{proposition}[statements]{Proposition}

\theoremstyle{definition}
\newtheorem{example}[statements]{Example}
\newtheorem{definition}[statements]{Definition}

\theoremstyle{remark}
\newtheorem{compactificationconstruction}[statements]{Compactification construction}
\newtheorem{remark}[statements]{Remark}

\newcommand{\ZZ}{{\mathbb Z}}

\newcommand{\PP}{{\mathbb P}}
\newcommand{\CC}{{\mathbb C}}
\newcommand{\RR}{{\mathbb R}}

\newcommand{\TT}{{\mathbb T}}
\newcommand{\FF}{{\mathbb F}}
\newcommand{\Aff}{{\mathbb A}}

\newfont{\smallskob}{cmbx7 scaled\magstep4}
\newfont{\bigskob}{cmbx12 scaled\magstep4}

\newcommand{\Pic}{\mathrm{Pic}\,}

\newcommand{\Spec}{\mathrm{Spec}\,}

\newcommand{\cO}{\mathcal{O}}
\newcommand{\vol}{\mathrm{vol}\,}

\begin{document}


\title{Calabi--Yau compactifications of toric Landau--Ginzburg models for smooth Fano threefolds}

\author{Victor Przyjalkowski}

\thanks{This work is supported by the Russian Science Foundation under grant 14-50-00005.}

\address{Steklov Mathematical Institute of Russian Academy of Sciences, 8 Gubkina street, Moscow 119991, Russia} %

\email{victorprz@mi.ras.ru, victorprz@gmail.com}


\maketitle

\begin{abstract}
We prove that smooth Fano threefolds have toric Landau--Ginzburg models.
More precise, we prove that their Landau--Ginzburg models, presented as Laurent polynomials,
admit 
compactifications to families of K3 surfaces, and we describe their fibers over infinity. 
We also give an explicit construction
of Landau--Ginzburg models for del Pezzo surfaces and any divisors on them.
\end{abstract}




\bigskip

\section{Introduction}
The Mirror Symmetry conjecture, one of the deepest recent ideas in mathematics, relates varieties
or certain families of varieties --- so called Landau--Ginzburg models.
It states that every variety has a dual object whose symplectic properties correspond to
algebraic ones for the original variety and, vice versa, whose algebraic properties
correspond to symplectic ones for the original variety.

There are several levels of Mirror Symmetry conjecture. 
In this paper we study \emph{Mirror Symmetry Conjecture of variations of Hodge structures}.
It relates the (symplectic) Gromov--Witten invariants of Fano varieties (which
count the expected numbers of (rational) curves lying on the varieties) with
periods of dual algebraic families. Givental suggested dual Landau--Ginzburg models and proved
Mirror Symmetry Conjecture of variations of Hodge structures for smooth toric varieties and Fano complete intersections
therein, see~\cite{Gi97}.
Another description of Mirror Symmetry for toric varieties treats it as a duality
between toric varieties (or polytopes that define them), see, for instance,~\cite{Ba93}.
From this point of view Laurent polynomials appear naturally as anticanonical sections of toric varieties.

Based on these considerations a notion of \emph{toric Landau--Ginzburg model} was proposed in~\cite{Prz13}.
A toric Landau--Ginzburg model for a smooth Fano variety $X$ of dimension~$n$ and a chosen divisor $D$ on it is a Laurent polynomial $f_X$
in $n$ variables satisfying three conditions. The period condition relates periods of the family of fibers of the map $f_X\colon (\CC^*)^n\to \CC$ to the generating series of one-pointed Gromov--Witten invariants of~$X$ depending on $D$. The second one,
the Calabi--Yau condition, says that the fibers of the map $f_X$ can be compactified to Calabi--Yau varieties.
This condition is motivated by the so called Compactification Principle (see~\cite[Principle 32]{Prz13}) stating that a fiberwise compactification
of the family of the fibers of the toric Landau--Ginzburg model is a Landau--Ginzburg model
from the point of view of Homological Mirror Symmetry. Finally the toric condition states that there
is a degeneration of $X$ to a toric variety $T$ whose fan polytope 
is the Newton polytope of $f_X$. This condition is motivated by the treatment of Mirror Symmetry
for toric varieties as a duality between polytopes and by the deformation invariance of the Gromov--Witten invariants.

The existence of toric Landau--Ginzburg models has been shown for Fano threefolds of Picard rank one
and for complete intersections in projective spaces (see~\cite{Prz13} and~\cite{ILP13}). Aside from that only partial results are known.
In particular, the existence of Laurent polynomials satisfying the period condition for all smooth Fano threefolds (and trivial, that is equal to zero, divisors on them) with very ample anticanonical classes was
shown in~\cite{fanosearch}, and the toric condition for them is checked in forthcoming papers~\cite{IKKPS} and~\cite{DHKLP}.

The main goal of the paper is to complete the proof of the existence of toric Landau--Ginzburg models for smooth Fano threefolds,
in other words, to check the Calabi--Yau condition for them.
There are $105$ families of smooth Fano threefolds; anticanonical classes of $98$ of them are very ample.
We consider ``good'' toric degenerations of these, that is, the degenerations to Gorenstein toric varieties.
Given a Fano threefold $X$ and its Gorenstein toric degeneration $T$ we denote by
$\Delta$ its fan polytope.
It is reflexive, which means that its dual polytope $\nabla$ is integral.
Let us assume that a Laurent polynomial $f_X$, whose Newton polytope is $\Delta$,
is of Minkowski type, i.\,e., its coefficients correspond to expansions of facets of $\Delta$ to Minkowski sums
of elementary polygons (see Definition~\ref{definition: Minkowski}).
Assume also that $f_X$ satisfies the period condition for $X$ and a trivial divisor.
(According to~\cite{fanosearch} and~\cite{CCGK16}, such polynomials exist for all $98$ smooth Fano threefolds with very ample anticanonical classes.) The family of fibers of the map given by $f_X$ is a one-dimensional linear subsystem
in the anticanonical linear system of a toric variety $T^\vee$ whose fan polytope is $\nabla$.
Since $\Delta$ is three-dimensional and reflexive, $T^\vee$ has a crepant resolution.
One of the members of the family given by $f_X$ is a boundary divisor of $T^\vee$. The base locus of the family is a union of smooth rational curves  due to the special origin of coefficients of $f_X$.
This enables one to resolve the base locus, and, cutting out the boundary divisor, to get a Calabi--Yau compactification.
Moreover, this gives a \emph{log Calabi--Yau compactification}, that is a compactification to a compact variety
over $\PP^1$ whose anticanonical divisor is a fiber.
In addition one gets a description of ``the fiber over infinity'' of the compactified Landau--Ginzburg model.

\begin{theorem}[Theorem~\ref{theorem:main}]
Any Minkowski Laurent polynomial in three variables admits a log Calabi--Yau compactification.
\end{theorem}

We also construct ``by hand'' Calabi--Yau compactifications for Laurent polynomials corresponding to seven Fano threefolds
whose anticanonical classes are not very ample.
More precisely, 
we compactify the families to families of (singular) quartics in $\PP^3$ or to hypersurfaces of bidegree $(2,3)$ in $\PP^1\times \PP^2$, which
have crepant resolutions. 

In the proof of Theorem~\ref{theorem:main} we do not use the fact that the variety $T$ is a toric degeneration of $X$.
This fact is proved in forthcoming papers~\cite{DHKLP} and~\cite{IKKPS} (see also~\cite{DH15}).
This, together with the results discussed above, gives the following corollary.

\begin{corollary}[Corollary~\ref{corollary: main}]
A pair of a smooth Fano threefold $X$ and a trivial divisor on it has a toric Landau--Ginzburg model.
Moreover, if $-K_X$ is very ample, then any Minkowski Laurent polynomial satisfying the period condition for $(X,0)$
is a toric Landau--Ginzburg model.
\end{corollary}

One can get log Calabi--Yau compactifications for del Pezzo surfaces (of degrees greater then $2$) in a similar way.
In this case the base locus is just a set of points (possibly with multiplicities),
and the resolution of the base locus is given just by a number of blow ups of these points.
Landau--Ginzburg models for del Pezzo surfaces equipped with general divisors
are constructed in~\cite{AKO06}; see also~\cite{UY13} for another description for arbitrary divisors. Our compactification procedure gives a precise method
to write down a Laurent polynomial (and to describe singularities of its fibers)
for \emph{any} chosen divisor.

\medskip

The paper is organized as follows. In Section~\ref{section:preliminaries} we recall some preliminaries and
give the main notation and definitions. In particular, we define toric Landau--Ginzburg models
and give the definition of Minkowski Laurent polynomials.
In Section~\ref{section: del Pezzo} we consider the two-dimensional case. We present an explicit way to write
down a toric Landau--Ginzburg model for a del Pezzo surface with an arbitrary divisor on it.
In Section~\ref{section:main theorem} we study the case of Fano threefolds and trivial divisors on them.
We describe the structure of the base loci of the families
we are interested in. Using this we prove the existence of Calabi--Yau compactifications for Minkowski Laurent
polynomials (Theorem~\ref{theorem:main}).
Then we construct Calabi--Yau compactifications for 
the rest cases. We summarize our considerations in Corollary~\ref{corollary: main}, where we state that
smooth Fano threefolds 
have toric Landau--Ginzburg models.
We also give some remarks and questions on 
fibers over infinity of their compactifications.

\medskip

{\bf Notation and conventions.}
All varieties are defined over the field $\CC$ of complex numbers.
We use the standard notation for multidegrees: given an element $a=(a_1,\ldots,a_k)\in \ZZ^k$
and a multivariable $x=(x_1,\ldots,x_k)$, we denote $x_1^{a_1}\ldots x_k^{a_k}$ by $x^a$.
For a variety $X$, we denote $\Pic (X)\otimes \CC$ by $\Pic(X)_{\CC}$.
For a Laurent polynomial $f\in \CC[\ZZ^n]$ we denote its Newton polytope, i.\,e. a
convex hull in $\ZZ^n\otimes \RR$ of exponents of non-zero monomials of $f$, by $N(f)$.
Smooth del Pezzo surface of degree $d$ (excluding the case of quadric surface) is denoted by $S_d$.
We denote by $X_{k-m}$ the smooth Fano variety (considered as an element of a family of Fano varieties of this type) of Picard rank $k$ and number $m$ in the lists from~\cite{IP99}.
The missing variety of Picard rank $4$ and degree $26$ we denote by $X_{4-2}$ and shift subscripts for varieties
of the same Picard rank and greater degree.
We denote by $\PP(a_0,\ldots,a_n)$ the weighted projective space with weights $a_0,\ldots,a_n$.
The (weighted) projective space with coordinates $x_0,\ldots,x_n$ we denote by $\PP[x_0:\ldots:x_n]$.
The affine space with coordinates $x_0,\ldots,x_n$ we denote by $\Aff[x_1,\ldots,x_n]$.
The ring $\CC[x_1^{\pm 1},\ldots,x_n^{\pm 1}]$ is denoted by $\TT[x_1,\ldots, x_n]$.
A polytope with integral vertices we call integral.
We threat a pencil in birational sense: a pencil for us is a family birational to a family of fibers of a morphism to $\PP^1$.

\medskip

{\bf Acknowledgements.}
The author is grateful to A.\,Corti, conversations with whom drastically improved the paper, and to
V.\,Golyshev, A.\,Harder, N.\,Ilten, A.\,Kasprzyk, V.\,Lunts, D.\,Sakovics, C.\,Shramov, and a referee
for helpful comments.

\section{Preliminaries}
\label{section:preliminaries}
\subsection{Toric geometry}
Consider a toric variety $T$.
\emph{A fan} (or \emph{spanning}) polytope $F(T)$ is a convex hull of integral generators
of fan's rays for $T$.
Let $\Delta=F(T)\subset N_\RR=\ZZ^n\otimes \RR$. 
Let
$$
\nabla=\{x\ |\ \langle x,y\rangle \geq -1 \mbox{ for all } y\in \Delta \}\subset M_\RR=N^\vee\otimes \RR
$$
be the dual polytope.

For an integral polytope $\Delta$ we associate a (singular) toric Fano variety $T_{\Delta}$
defined by a fan whose cones are cones over faces of $\Delta$. We also associate a (not uniquely defined) toric variety $\widetilde{T}_\Delta$
with $F(\widetilde T_\Delta)=\Delta$ such that for any toric variety $T'$ with $F(T')=\Delta$ and for any
morphism $T'\to \widetilde T_\Delta$ one has $T'\simeq \widetilde T_\Delta$. In other words, $\widetilde T_\Delta$ is
given by ``maximal triangulation'' of $\Delta$.

In the rest of the paper we assume that $\nabla$ is integral, in other words, that $\Delta$ (or $\nabla$, or $T$, or $T_\nabla$, or
$\widetilde T_\nabla$) is \emph{reflexive}.
In particular this means that integral points of both $\Delta$ and $\nabla$ are either the origin or lie on the boundary.
We denote $T_\nabla$ by $T^\vee$ and $\widetilde T_\nabla$ by $\widetilde T^\vee$.


\begin{lemma}
\label{lemma: 3dim smoothness}
Let $T$ be a threefold reflexive toric variety. Then $\widetilde T^\vee$ is smooth.
\end{lemma}

\begin{proof}
Let $C$ be a two-dimensional cone of the fan of $\widetilde T^\vee$. It is a cone over an integral triangle $R$ without strictly internal integral points, such that $R$ lies in the affine plane $L=\{x|\ \langle x,y\rangle\geq -1\}$ for some $y\in N$.
This means that is some basis $e_1,e_2,e_3$ in $M$ one
gets $L=\{a_1e_1+a_2e_2+e_3\}$. Let $P$ be a pyramid over $R$ whose vertex is an origin. Then by Pick's formula
one has $\vol R=\frac{1}{2}$ and $\vol P=\frac{1}{6}$, which means that vertices of $R$ form basis in $M$, so $\widetilde T^\vee$ is smooth.
\end{proof}

Unfortunately, Lemma~\ref{lemma: 3dim smoothness} does not hold for higher dimensions in general, because there are $n$-dimensional
simplices whose only integral points are vertices, such that their volumes are greater then $\frac{1}{n!}$.

\begin{definition}[see~\cite{CCGK16}]
\label{definition: Minkowski}
An integral polygon is called \emph{of type $A_n$}, $n\geq 0$, if it is a triangle such that two its edges have integral length $1$ and
the rest one has integral length $n$.
(In other words, its integral points are $3$ vertices and $n-1$ points lying on the same edge.)
In particular, $A_0$ is a segment of integral length $1$.

An integral polygon $P$ is called \emph{Minkowski}, or \emph{of Minkowski type}, if it is a Minkowski sum of several polygons of type $A_n$,
that is
$$
P=\left\{p_1+\ldots +p_k |\, p_i\in P_i \right\}
$$
for some polygons $P_i$ of type $A_{k_i}$, and if the affine lattice generated by $P\cap \ZZ^2$ is a sum of affine lattices
generated by $P_i\cap \ZZ^2$. Such decomposition is called \emph{admissible lattice Minkowski decomposition} and
denoted by $P=P_1+\ldots +P_k$.

An integral three-dimensional polytope is called \emph{Minkowski} if it is reflexive and if all its facets
are Minkowski polygons.
\end{definition}

Finally summarize some facts related to toric varieties and their anticanonical sections. One can see, say,~\cite{Da78} for details.
It is more convenient to start from the toric variety $T^\vee$ for the following.

\begin{itemize}
\item[Fact $1$.]
Let the anticanonical class $-K_{T^\vee}$ be very ample (in particular, this holds in reflexive threefold case, see~\cite{JR06} and~\cite{CPS05}).
One can embed $T^\vee$ to a projective space in the following way.
Consider a set $A\subset M$ of integral points in a polytope $\Delta$ dual to $F(T^\vee)$. Consider a projective space
$\PP$ whose coordinates $x_i$ correspond to elements $a_i$ of $A$.
Associate a homogenous equation $\prod x_i^{\alpha_i}=\prod x_j^{\beta_j}$ with any homogenous relation $\sum \alpha_i a_i=\sum \beta_j a_j$, $\alpha_i, \beta_j\in \ZZ_+$.
The variety $T^\vee$ is cut out in $\PP$ by equations associated to all homogenous relations on $a_i$.

\item[Fact $2$.] The anticanonical linear system of $T^\vee$ is a restriction of $\cO_\PP(1)$. In particular,
it can be described as (a projectivisation of) a linear system of Laurent polynomials whose Newton
polytopes contain in $\Delta$.

\item[Fact $3$.]
Toric strata of $T^\vee$ of dimension $k$ correspond to $k$-dimensional faces of $\Delta$.
Denote by $R_{f}$ an anticanonical
section corresponding to a Laurent polynomial $f\in \CC[N]$ and by $F_Q$ a strata corresponding to a face $Q$ of $\Delta$.
Denote by $f|_Q$ a sum of those monomials of $f$ whose support lie in $Q$.
Denote by $\PP_Q$ a projective space whose coordinates correspond to $Q\cap N$.
(In particular, $Q$ is cut out in $\PP_Q$ by homogenous relations on integral points of $Q\cap N$.)
Then $R_{Q,f}=R_f|_{F_Q}=\{f|_Q=0\}\subset \PP_Q$.

\item[Fact $4$.]
In particular, $R_{f}$ does not pass through
a toric point corresponding to a vertex of $\Delta$ if and only if its coefficient
at this vertex is non-zero. The constant Laurent polynomial corresponds to
the boundary divisor of $T^\vee$.
\end{itemize}

\subsection{Laurent polynomials and toric Landau--Ginzburg models}
Different polynomials with the same Newton polytope $\Delta$ give the same toric variety $T_\Delta$. However a choice of particular
coefficients of a polynomial is important from Mirror Symmetry point of view.

We briefly recall a notion of toric Landau--Ginzburg model. More details see, for instance, in~\cite{Prz13}.

Let $X$ be a smooth Fano variety of dimension $n$ and Picard number $\rho$.
Let the number
$$
\langle \tau_{a_1} \gamma_1, \ldots,\tau_{a_k} \gamma_n\rangle_\beta, \quad a_i\in \ZZ_{\geq 0},\ \gamma_i \in H^*(X,\CC),\ \beta \in H_{2}(X,\ZZ),
$$
be a \emph{$k$-pointed genus $0$ Gromov--Witten invariant with descendants} for $X$, see~\cite[VI-2.1]{Ma99}.

Choose $\CC$-divisors $L$ and $D$ on $X$.
One can treat $L$ as a direction (one-dimensional subtorus) on $\TT=\Spec \CC[H_2(X,\ZZ)]$ and $D$ as a point $p_{D}$ on $\TT$.
Let $\mathbf 1$ be the fundamental class of $X$.
Denote by $R$ a set of classes of effective curves.
%
The series
$$
\widetilde{I}^{X,L,D}_{0}(t)
=1+\sum_{\beta \in R,\ a\in \ZZ_{\geq 0}} (\beta\cdot L)!\langle\tau_{a} \mathbf 1\rangle_{\beta}
\cdot e^{-\beta\cdot D}t^{\beta\cdot L}
$$
is called \emph{a constant term of regularized $I$-series} for $X$. It is a flat section
for the restriction of the second Dubrovin's connection 
on $\TT$ to a direction of $D$ passing through $p_{D}$.

Now we are interested in an \emph{anticanonical direction}, so $L=-K_X$.
We denote 
$\widetilde{I}_0^{X,-K_X,D}$ by $\widetilde{I}_0^{X,D}$.
The latter series is the main object we associate to the pair $(X,D)$.
Moreover, we often consider the case $D=0$; in this case we use notation 
$\widetilde I^X_0$ for 
$\widetilde I^{X,0}_0$. 

Now consider the other side of Mirror Symmetry.

Let~$\phi[f]$ be a constant term of a Laurent polynomial~$f$.
Define $I_f(t)=\sum \phi[f^j] t^j$.
The following theorem (which is a mathematical folklore, see~\cite[Proposition 2.3]{Prz08} 
for the proof)
justifies this definition.

\begin{theorem}
\label{theorem: Picard--Fuchs}
Let $f$ be a Laurent polynomial in $n$ variables.
Let $P$ be a Picard--Fuchs differential operator for a pencil of hypersurfaces in a torus
provided by $f$.
Then
one has~\mbox{$P[I_f(t)]=0$}.
\end{theorem}

\begin{definition}[{see~\cite[\S6]{Prz13}}]
\label{definition: toric LG}
\emph{A toric Landau--Ginzburg model} for a pair of a smooth Fano variety $X$ of dimension $n$ and divisor $D$ on it is a Laurent polynomial $\mbox{$f\in \TT[x_1, \ldots, x_n]$}$ which satisfies the following.
\begin{description}
  \item[Period condition] One has $I_f=\widetilde{I}^{X,D}_0$.
  \item[Calabi--Yau condition] There exists a relative compactification of a family
$$f\colon (\CC^*)^n\to \CC$$
whose total space is a (non-compact) smooth Calabi--Yau
variety $Y$. Such compactification is called \emph{a Calabi--Yau compactification}.
  \item[Toric condition] There is a degeneration
   $X\rightsquigarrow T_X$ to a toric variety~$T_X$ such that $F(T_X)=N(f)$.
\end{description}

Laurent polynomial satisfying the period condition is called a \emph{weak Landau--Ginzburg model}.
\end{definition}

\begin{definition}
A compactification of the family $f\colon (\CC^*)^n\to \CC$ to a family $f\colon Z\to \PP^1$, where $Z$ is smooth and
$-K_Z=f^{-1}(\infty)$, is called a \emph{log Calabi--Yau compactification} (cf. a notion of tame compactified Landau--Ginzburg model in~\cite{KKP14}).
\end{definition}

We consider Mirror Symmetry as a correspondence between Fano varieties and Laurent polynomials.
That is, a strong version of Mirror Symmetry of variations of Hodge structures conjecture states the following.
\begin{conjecture}[{see~\cite[Conjecture 38]{Prz13}}]
\label{conjecture:MS}
Any pair of a smooth Fano variety and a divisor on it has a toric Landau--Ginzburg model. 
\end{conjecture}

In addition to the threefold case discussed in the introduction, the existence of Laurent polynomials satisfying
period condition for smooth toric varieties is shown in~\cite{Gi97},
for complete intersections in Grassmannians is shown in~\cite{PSh14},~\cite{PSh14b},~\cite{PSh15b},
for some complete intersections in some toric varieties is shown in~\cite{CKP14} and~\cite{DH15}.
In~\cite{DH15} the toric condition for some complete intersections in toric varieties and partial
flag varieties is also checked.

In a lot of cases 
polynomials satisfying the period and toric conditions satisfy
the Calabi--Yau condition as well. However it is not easy to check this condition:
there are no general enough approaches as for the rest two conditions are; usually one needs to
check the Calabi--Yau condition ``by hand''. The natural idea is to compactify the fibers
of the map $f\colon (\CC^*)^n\to \CC$ using the embedding $(\CC^*)^n\hookrightarrow T^\vee$,
where a toric variety $T=T_\Delta$ corresponds to $\Delta=N(f)$.
Indeed, the fibers compactify to anticanonical sections in $T^\vee$, and, since, have trivial canonical classes. However, first,
$T^\vee$ is usually singular, and, even if we resolve it (if it has a crepant resolution!),
we can just conclude that its general anticanonical section is a smooth Calabi--Yau variety,
but it is hard to say anything about the particular sections we need.
Second, the family of anticanonical sections we are interested in has a base locus which
we need to blow up to construct a Calabi--Yau compactification; and this blow up can be non-crepant.

Coefficients of the polynomials that correspond to trivial divisors tend to have very symmetric
coefficients, at least for the simplest toric degenerations. In this case the base loci
are more simple and enable us to construct Calabi--Yau compactifications.

Consider a Laurent polynomial $f$. Recall that we can construct a Fano toric variety $T=T_\Delta$, where
$\Delta=N(F)$, a dual toric Fano variety $T^\vee=T_\nabla$,  where $\nabla$ is a dual
polytope for $\Delta$, and maximally triangulated toric variety $\widetilde T^\vee=\widetilde T_\nabla$.
We make two assumptions for $f$. First, we assume that $\nabla$ is integral, in other words, that $\Delta$ is {reflexive}.
In particular this means that integral points of both $\Delta$ and $\nabla$ are either the origin or lie on the boundary.
The other assumption is related to the special ``symmetric'' choice of coefficients of $f$.

\begin{definition}[see~\cite{CCGK16}]
Let $P\in \ZZ^2\otimes \RR$ be an integral polygon of type $A_n$. Let $v_0,\ldots,v_n$ be consecutive integral points
on the edge of $P$ of integral length $n$ and let $u$ be the rest integral point of $P$.
Let $x=(x_1,x_2)$ be a multivariable that corresponds to an integral lattice $\ZZ^2\subset \RR^2$.
Put
$$
f_P=x^u+\sum \binom{n}{k} x^{v_k}.
$$
(In particular one has $f_P=x^u+x^{v_0}$ for $n=0$.)

Let $Q=Q_1+\ldots+Q_s$ be an admissible lattice Minkowski decomposition of an integral polygon $Q\subset \RR^2$.
Put
$$
f_{Q_1,\ldots,Q_s}=f_{Q_1}\cdot\ldots\cdot f_{Q_s}.
$$

A Laurent polynomial $f\subset \TT[x_1,x_2,x_3]$ is called \emph{Minkowski} if $N(f)$ is Minkowski and
for any facet $Q\subset N(f)$, 
as for an integral polygon, there exist an admissible lattice Minkowski decomposition $Q=Q_1+\ldots+Q_s$ such that $f|_Q=f_{Q_1,\ldots,Q_s}$.
\end{definition}

\begin{remark}[cf.~\cite{CCGK16}]
\label{remark: cDV threefolds}
There are $105$ families of smooth Fano threefolds (see~\cite{Is77},~\cite{Is78}, and~\cite{MM82}).
Among them $98$ families
have Minkowski weak Landau--Ginzburg models,
and seven varieties, that are $X_{1-1}$, $X_{1-11}$, $X_{2-1}$, $X_{2-2}$, $X_{2-3}$, $X_{9-1}$, and $X_{10-1}$,
do not have reflexive toric degenerations.
\end{remark}

\section{Del Pezzo surfaces}
\label{section: del Pezzo}

We start the section by recalling well known facts about del Pezzo surfaces. We refer, say, to~\cite{Do12} as to one of huge amount
of references on del Pezzo surfaces.

The initial definition of del Pezzo surface is the following one given by P.\,del Pezzo himself.

\begin{definition}[\cite{dP87}]
\emph{A del Pezzo surface} is a non-degenerate irreducible linear normal (that is it is not a projection of degree $d$ surface in $\PP^{d+1}$) surface in $\PP^d$ of degree $d$ which is not a cone.
\end{definition}

In modern words this means that a del Pezzo surface is an (anticanonically embedded) surface with ample anticanonical class
and canonical (the same as du Val, simple surface, Kleinian, or rational double point) singularities. (Classes of canonical and
Gorenstein singularities for surfaces coincide.)
So we use the following more general definition.

\begin{definition}
\emph{A del Pezzo surface} is a complete surface with ample anticanonical class and canonical singularities.
\emph{A weak del Pezzo surface} is a complete surface with nef and big anticanonical class and canonical singularities.
\end{definition}

\begin{remark}
Weak del Pezzo surfaces are (partial) minimal resolutions of singularities of del Pezzo surfaces. Exceptional
divisors of the resolutions are $(-2)$-curves.
\end{remark}

\emph{A degree} of del Pezzo surface $S$ is the number $d=(-K_S)^2$. One have $1\leq d \leq 9$.
If $d>2$, then the anticanonical class of $S$ is very ample and it gives the embedding $S\hookrightarrow \PP^d$, so both definitions
coincide.
From now on we assume that $d>2$.

Obviously, projecting a degree $d$ surface in $\PP^d$ from a point on it one gets degree $d-1$ surface in $\PP^{d-1}$.
This projection is nothing but blow up of the center of the projection and blow down all lines passing through the
point. (By adjunction formula these lines are $(-2)$-curves.) If we choose general (say, not lying on lines) centers of projections we get a classical description
of a smooth del Pezzo surface of degree $d$ as a quadric surface (with $d=8$) or a blow up of $\PP^2$ in $9-d$ points.
They degenerate to singular surfaces which are projections from non-general points (including infinitely close ones).
Moreover, all del Pezzo surfaces of given degree lie in the same irreducible deformation space except for degree $8$ when there are two
components (one for a quadric surface and one for a blow up $\FF_1$ of $\PP^2$). General elements of the families are smooth,
and all singular del Pezzo surfaces are degenerations of smooth ones in these families.
This description enables us to construct toric degenerations of del Pezzo surfaces.
That is, $\PP^2$ is toric itself. 
Projecting from toric points one gets a (possibly singular) toric del Pezzo surfaces.

\begin{remark}
Del Pezzo surfaces of degree $1$ or $2$ also have toric degenerations.
Indeed, these surfaces can be described as hypersurfaces in weighted projective spaces, that is ones
of degree $4$ in $\PP(1,1,1,2)$ and of degree $6$ in $\PP(1,1,2,3)$ correspondingly,
so they can be degenerated to binomial hypersurfaces.
However their singularities are worse then canonical.
\end{remark}

Let $T_S$ be a Gorenstein toric degeneration of a del Pezzo surface $S$ of degree $d$. Let $\Delta=F(T_S)\subset N_\RR=\ZZ^2\otimes \RR$ be a fan polygon
of $T_S$. Let $f$ be a Laurent polynomial such that $N(f)=\Delta$.

Our goal now is to describe in details a way to construct a Calabi--Yau compactification for $f$.
More precise, we construct a commutative diagram
\[\xymatrix{
(\CC^*)^2\ar@{^{(}->}[r] \ar[rd]^f& Y\ar[d] \ar@{^{(}->}[r] & Z\ar[d]\\
& \Aff^1 \ar@{^{(}->}[r] & \PP^1,\\
}
\]
where $Y$ and $Z$ are smooth, fibers of maps $Y\to \Aff^1$ and $Z\to \PP^1$ are compact, and $-K_Z=f^{-1}(\infty)$;
we denote all ``vertical'' maps in the diagram by $f$ for simplicity.

The strategy is the following. First we consider a natural compactification of the pencil $\{f=\lambda\}$ to an elliptic pencil
in a toric del Pezzo surface $T^\vee$. Then we resolve singularities of $T^\vee$ and get a pencil in a smooth toric weak del Pezzo surface
$\widetilde{T}^\vee$. Finally we resolve a base locus of the pencil to get $Z$. We get $Y$ cutting out strict transform of the boundary divisor of $\widetilde{T}^\vee$.

The polygon $\Delta$ 
has integral vertices in
$N_\RR$ and it has the origin as a unique strictly internal integral point.
A dual polygon $\nabla=\Delta^\vee \subset M=N^\vee$ has integral vertices and a unique strictly internal integral point as well. 
Geometrically this means that singularities of $T$ and $T^\vee$ are canonical. 

\begin{remark}
The normalized volume of $\nabla$ is given by
$$\mathrm{vol}\, \nabla= |\mbox{integral points in $\nabla$}|-1=(-K_{S})^2=d.$$
It is easy to see 
that $$|\mbox{integral points in $\Delta$}|+|\mbox{integral points in $\nabla$}|=12.$$
In particular, $\mathrm{vol}\, \Delta=12-d$.
\end{remark}

\begin{compactificationconstruction}
\label{compactification construction}
By Fact $2$, the anticanonical linear system on $T^\vee$ can be described as a projectivisation of a linear space of Laurent polynomials whose
Newton polygons are contained in $\nabla^\vee=\Delta$. Thus the natural way to compactify the family is to do it using
embedding $(\CC^*)^2\hookrightarrow T^\vee$. Fibers of the family are anticanonical divisors in this (possibly singular)
toric variety. Two anticanonical sections intersect by $(-K_{T^\vee})^2= \mathrm{vol}\, \Delta=12-d$ points (counted with multiplicities),
so the compactification of the pencil in $T^\vee$ has $12-d$ base points (possibly with multiplicities).
The pencil $\{\lambda_0 f=\lambda_1\}$, $(\lambda_0:\lambda_1)\in \PP$
is generated by its general member and a divisor corresponding to a constant Laurent polynomial, i.\,e. to the boundary
divisor of $T^\vee$. Let us mention that the torus invariant points of $T^\vee$ do not lie in the base locus of the family
by Fact $4$.

Let $\widetilde{T}^\vee\to T^\vee$ be a minimal resolution of singularities of $T^\vee$.
Pull back the pencil under consideration. We get an elliptic pencil with $12-d$ base points (with multiplicities),
which are smooth points of 
the boundary divisor $D$ 
of the toric surface $\widetilde{T}^\vee$; this divisor is a wheel of $d$ smooth rational curves.
Blow up these base points and get an elliptic surface $Z$.
Let $E_1,\ldots,E_{12-d}$ be the exceptional curves of the blow up $\pi\colon Z\to \widetilde{T}^\vee$; in particular, $Z$ is not toric. Denote strict transform
of $D$ by $D$ for simplicity.
Then one has
$$
-K_Z=\pi^*(-K_{\widetilde{T}^\vee})-\sum E_i=D+\sum E_i-\sum E_i=D.
$$
Thus the anticanonical class $-K_Z$ contains $D$ and consists of fibers of $Z$.
This, in particular, means that an open variety $Y=Z\setminus D$ is a Calabi--Yau compactification of the pencil provided by $f$.
This variety has $e>0$ sections, where $e$ is a number of base points of the pencil in $\widetilde{T}^\vee$ counted \emph{without} multiplicities.

Summarizing, we obtain an elliptic surface $f\colon Z\to \PP^1$ with smooth total space $Z$ and a wheel $D$ of $d$ smooth rational curves over $\infty$.
\end{compactificationconstruction}

\begin{remark}
\label{remark: Hodge numbers for Z}
Let the polynomial $f$ be general among ones with the same Newton polygon. Then singular fibers of $Z\to \PP^1$ are either curves
with a single node or a wheel of $d$ rational curves over $\infty$. By Noether formula one has
$$
12\chi (\cO_Z)=(-K_Z)^2+e(Z)=e(Z),
$$
where $e(Z)$ is a topological Euler characteristic. Thus singular fibers for $Z\to \PP^1$
are $d$ curves with one node and a wheel of $d$ curves over $\infty$. This description is given in~\cite{AKO06}.
\end{remark}



\begin{remark}
One can compactify all toric Landau--Ginzburg models for all del Pezzo surfaces of degree at least three simultaneously.
That is, all reflexive polygons are contained in the biggest polygon $B$, that has vertices $(2,-1)$, $(-1,2)$, $(-1,-1)$. Thus fibers of all toric Landau--Ginzburg models can be simultaneously compactified to (possibly
singular) anticanonical curves on $T_{B^\vee}=\PP^2$. Blow up the base locus to construct a base points free family. However in this
case a general member of the family can pass through toric points as it can happen that $N(f)\varsubsetneq B$.
This means that some of exceptional divisors of the minimal resolution are 
extra curves in a wheel over infinity.

In other words, consider a triangle of lines 
on $\PP^2$.
A general member of the pencil given by $f$ is an elliptic curve on $\PP^2$.
The total space of the log Calabi--Yau compactification 
is a blow up of nine intersection points (counted with multiplicities) of the elliptic curve and the triangle of lines.
Exceptional divisors for points lying over vertices of the triangle are components of the wheel over infinity for the log
Calabi--Yau compactification; the others are either sections of the pencil or components of fibers over finite points.
\end{remark}

Now describe toric Landau--Ginzburg models for del Pezzo surfaces and toric weak del Pezzo surfaces.
That is, for a del Pezzo surface $S$, its
Gorenstein toric degeneration $T$ with a fan polygon $\Delta$,
its crepant resolution $\widetilde{T}$ with the same fan polygon, and a divisor $D\in \Pic(S)_\CC\cong \Pic(\widetilde T)_\CC$,
we construct two Laurent polynomials $f_{S,D}$ and $f_{\widetilde T,D}$, that are toric Landau--Ginzburg models
for $S$ and $\widetilde T$ correspondingly, by induction. For this use, in particular, Givental's construction of Landau--Ginzburg
models for smooth toric varieties, see~\cite{Gi97}.

Let $S\cong \PP^1\times \PP^1$ be a quadric surface, and let $D_S$ be an $(a,b)$-divisor on it.
Let $T_1=S$,
and let $T_2$ be a quadratic cone; $T_1$ and $T_2$ are the only Gorenstein toric degenerations of $S$.
The crepant resolution of $T_2$ is a Hirzebruch surface $\FF_2$, so let $D_{\FF_2}=\alpha s+\beta f$,
where $s$ is a section of $\FF_2$, so that $s^2=-2$, and let $f$ be a fiber of the map $\FF_2\to \PP^1$.
Define
$$
f_{S,D_S}=f_{\widetilde T_1,D_S}=x+\frac{e^{-a}}{x}+y+\frac{e^{-b}}{y}
$$
for the first toric degeneration and
$$
f_{S,D_S}=y+e^{-a}\frac{1}{xy}+\left(e^{-a}+e^{-b}\right)\frac{1}{y}+e^{-b}\frac{x}{y},\ \ f_{\widetilde T_2,D_{\FF_2}}=y+\frac{e^{-\beta}}{xy}
+\frac{e^{-\alpha}}{y}+\frac{x}{y}
$$
for the second one.

Now assume that $S$ is a blow up of $\PP^2$.
First let $S=T=\widetilde T=\PP^2$, let $l$ be a class of a line on $S$, and let $D=a_0l$.
Then up to a toric change of variables one has
$$
f_{\PP^2,D}=x+y+\frac{e^{-a_0}}{xy}.
$$
Now let $S'$ be a blow up of $\PP^2$ in $k$ points with exceptional
divisors $e_1, \ldots, e_k$, let $S$ be a blow up of $S'$ in a point,
and let $e_{k+1}$ be an exceptional divisor for the blow up.
We identify divisors on $S'$ and their strict transforms on $S$,
so $\Pic(S')=\Pic (\widetilde T')=\ZZ l+\ZZ{e_1}+\ldots+\ZZ{e_k}$ and
$\Pic(S)=\ZZ l+\ZZ{e_1}+\ldots+\ZZ{e_k}+\ZZ{e_{k+1}}$.
Let $D'=a_0l+a_1e_1+\ldots+a_ke_k\in \Pic(S')_\CC$
and $D=D'+a_{k+1}e_{k+1}\in \Pic(S)_\CC$.
First describe the polynomial $f_{\widetilde T,D}$.
Combinatorially $\Delta=F(\widetilde T)$
is obtained from a polygon $\Delta'=F(\widetilde T')$ by adding
one integral point $K$ that corresponds to the exceptional
divisor $e_{k+1}$, and taking a convex hull. Let
$L$, $R$ be boundary points of $\Delta$ neighbor to $K$, left and right with respect to the clockwise order.
Let $c_L$ and $c_R$ be coefficients in $f_{\widetilde T',D'}$
at monomials corresponding to $L$ and $R$. Let $M\in \TT[x,y]$ be a monomial corresponding to $K$.
Then from Givental's description of Landau--Ginzburg models for toric varieties (see~\cite{Gi97}) one gets
$$
f_{\widetilde T,D}=f_{\widetilde S',D'}+c_Lc_Re^{-a_{k+1}}M.
$$

The polynomial $f_{S,D}$ differs from $f_{\widetilde T,D}$ by coefficients at non-vertex boundary points.
For any boundary point $K\subset \Delta$ define \emph{marking} $m_K$ as a coefficient of $f_{\widetilde T,D}$ at $K$.
Consider a facet of $\Delta$ and let $K_0,\ldots, K_r$ be integral points in clockwise order of this facet.
Then coefficient of $f_{S,D}$ at $K_i$ is a coefficient at $s^i$ in the polynomial
$$
m_{K_0}\left(1+\frac{m_{K_1}}{m_{K_0}}s\right)\cdot\ldots\cdot \left(1+\frac{m_{K_r}}{m_{K_{r-1}}}s\right).
$$

\begin{remark}
One has $\Pic(S)\cong \Pic(\widetilde T)$.
That is, if $S$ is not a quadric, then both $S$ and $\widetilde T$
are obtained by a sequence of blow ups in points (the only difference is
that the points for $\widetilde T$ can lie on exceptional divisors of previous blow ups).
Thus in both cases Picard groups are generated by a class of a line on $\PP^2$ and exceptional divisors
$e_1,\ldots, e_k$. However an image of $e_i$ under the map of Picard groups
given by the degeneration of $S$ to $\widetilde T$ can be not equal to $e_i$ itself
but to some linear combination of the exceptional divisors.
In other words these bases do not agree with the degeneration map.
\end{remark}

\begin{remark}
The spaces parameterizing toric Landau--Ginzburg models for $S$ and for $\widetilde T$ are the same ---
they are the spaces of Laurent polynomials with Newton polygon $\Delta$ modulo toric rescaling.
Thus any Laurent polynomial correspond to different elements of $\Pic(S)_\CC\cong \Pic(\widetilde T)_\CC$.
This gives a map $\Pic(S)_{\CC}\to \Pic(\widetilde T)_{\CC}$. However this
map is transcendental because of exponential nature of the parametrization.
\end{remark}

\begin{proposition}
\label{proposition: toric LG for del Pezzo}
The Laurent polynomial $f_{S,D}$ is a toric Landau--Ginzburg model for $(S,D)$.
\end{proposition}

\begin{proof}
It is well known that $S$ is either a smooth toric variety
or a complete intersection in a smooth toric variety.
This enables one to compute a series $\widetilde I^S$ and, since, $\widetilde I^{S,D}$ following~\cite{Gi97}.
Using this it is straightforward to check that the period condition for $f_{S,D}$ holds.
The Calabi--Yau condition holds by Compactification construction~\ref{compactification construction}. 
Finally the toric condition holds by construction.
(See Example~\ref{example: S7}.)
\end{proof}

\begin{proposition}
\label{proposition:mutations}
Consider two different Gorenstein toric degenerations $T_1$ and $T_2$ of a del Pezzo surface $S$. Let $\Delta_1=F(T_1)$ and $\Delta_2=F(T_2)$.
Consider families of Calabi--Yau compactifications of Laurent polynomials
with Newton polygons $\Delta_1$ and $\Delta_2$. Then there is a birational isomorphism of these families.
In other words, there is a birational isomorphism between affine spaces of Laurent polynomials
with supports in $\Delta_1$ and $\Delta_2$ modulo toric change of variables that preserves Calabi--Yau compactifications.
\end{proposition}

\begin{proof}
One can check that polygons $\Delta_1$ and $\Delta_2$ differ by (a sequence of) mutations (see, say,~\cite{ACGK12}).
These mutations agree with fiberwise birational isomorphisms of toric Landau--Ginzburg models
modulo change of basis in $H^2(S,\ZZ)$ by the construction. The statement follows from the fact that birational elliptic curves are isomorphic.
\end{proof}

\begin{remark}
\label{remark:canonical_the_same}
Let $D=0$. Then
the polynomial $f_{S,0}$ has
coefficients $1$ at vertices of its Newton polygon and $\binom{n}{k}$ at $k$-th
integral point of an edge of integral length $n$.
In other words, $f_{S,0}$ is binomial. 
\end{remark}

\begin{example}
\label{example: S7}
Let $S=S_7$. This surface has two Gorenstein toric degenerations: it is toric itself, and also it can be degenerated to
a singular surface which is obtained by a blow up of $\PP^2$, a blow up of a point on the exceptional curve, and a blow down
the first exceptional curve to a point of type $A_2$.

Let $\Delta_1$ be the polygon with vertices $(1,0)$, $(1,1)$, $(0,1)$, $(-1,-1)$, $(0,-1)$,
and let $D=a_0l+a_1e_1+a_2e_2$.
Then
$$
f_{\widetilde T_{\Delta_1},D}=f_{S,D}=x+y+e^{-a_0}\frac{1}{xy}+e^{-(a_0+a_1)}\frac{1}{y}+e^{-a_2}xy.
$$

Let $\Delta_2$ be the polygon with vertices $(1,0)$, $(0,1)$, $(-1,-1)$, $(1,-1)$,
and let $D=a_0l+a_1e_1+a_2e_2$.
Then
$$
f_{\widetilde T_{\Delta_2},D}=x+y+e^{-a_0}\frac{1}{xy}+e^{-(a_0+a_1)}\frac{1}{y}+e^{-(a_0+a_1+a_2)}\frac{x}{y},
$$
$$
f'_{S,D}=x+y+e^{-a_0}\frac{1}{xy}+\left(e^{-(a_0+a_1)}+e^{-(a_0+a_2)}\right)\frac{1}{y}+e^{-(a_0+a_1+a_2)}\frac{x}{y}.
$$
(Here $f_{S,D}$ and $f'_{S,D}$ are toric Landau--Ginzburg models for $(S,D)$ in different bases in $(\CC^*)^2$.)
One can easily check that the mutation
$$
x\to x,\ \ \ y\to \frac{y}{1+e^{-a_2}x}
$$
sends $f_{S,D}$ to $f'_{S,D}$.

The surface $S$ is toric, so by Givental
$$
\widetilde I_0^{S,D}= \sum_{k,l,m}  \frac{(2k+3l+2m)!e^{-a_0(k+l+m)-a_1k-a_2m}t^{2k+3l+2m}}{(k+l)!(l+m)!k!l!m!}
$$
(see~\cite{fanosearch}).
One can check that $\widetilde I_0^{S,D}=I_{f_{S,D}}=I_{f'_{S,D}}$.
\end{example}

\section{Minkowski toric Landau--Ginzburg models}
\label{section:main theorem}

\begin{lemma}
\label{lemma:restrictions Minkowski}
Let $f$ be a Minkowski Laurent polynomial.
Then for any face $Q$ of $\Delta$ the curve $R_{Q,f}$ is a union of (transversally intersecting) smooth rational curves (possibly with multiplicities).
\end{lemma}

\begin{proof}
Let $Q$ be of type $A_k$, $k>0$.
In appropriate basis $Q$ has vertices $u=(0,1)$, $v_0=(0,0)$, $v_k=(k,0)$, and integral points $v_i=(i,0)$.
Let $x,x_0,\ldots, x_k$ be coordinates corresponding to $u, v_0,\ldots, v_k$. Then, according to Fact $1$, $F_Q$ is given
by relations $x_ix_j=x_rx_s$, $i+j=r+s$, in $\PP[x:x_0:\ldots:x_k]$. This means that
$F_Q=v_k\left(\PP(1,1,k)\right)$ is an image of $k$-th Veronese map of $\PP(1,1,k)$. Let $y_0,y_1,y_2$ be coordinates
on $\PP(1,1,k)$, where a weight of $y_2$ is $k$. One has $R_{Q,f}=\{\sum x_i \binom{n}{i}+x=0\}\cap F_Q\subset \PP^Q$. Hence
$$
R_{Q,f}=\{(y_0+y_1)^k+y_2=0\}\subset \PP(1,1,k),
$$ so $R_{Q,f}$ projects isomorphically to $\PP^1$ under projection of $\PP(1,1,k)$ on $\PP^1$
along the third coordinate. Thus $R_{Q,f}$ is a smooth rational curve with multiplicity one.

Now let $Q=Q_1+\ldots+Q_n$ be an admissible lattice Minkowski decomposition, where $Q_i$ is of type $A_{k_i}$, such that $f|_Q=f_{Q_1}\cdot \ldots\cdot f_{Q_n}$.
Then, as above, there are Veronese embeddings $v_{k_i}\colon \PP(1,1,k_i)\to \PP^{k_i+1}$, where $\PP^{k_i+1}$ are different projective
spaces. 
Let $\Pi$ be a product of these projective spaces over all 
$Q_i$, so that coordinates in $\Pi$ can be described as collections of integral points in $(Q_1,\ldots,Q_n)$. Denote the map $F_{Q_1}\times\ldots\times F_{Q_n}\to \Pi$ by $\varphi$. Let $\psi\colon \Pi \to \PP_S$ be a Segre embedding.
Let $\PP$ be the projective space whose
coordinates correspond to integral points in $Q$.
Let $x_{b_1,\ldots,b_n}$ be natural coordinates in $\PP_S$.
The space $\PP$ can be described as a linear section of $\PP_S$ that cuts out by the linear space
$$
L=\{x_{b_1,\ldots,b_n}=x_{b'_1,\ldots,b'_n}\mid b_1+\ldots+b_n=b'_1+\ldots+b'_n\},
$$
and $F_Q=\psi \varphi (F_{Q_1}\times\ldots\times F_{Q_n})\cap L$ in $\PP_S$. This gives birational isomorphisms
$F_{Q_i}\to F_Q$ for $k_i>0$ and $\PP^1$-bundles for $k_i=0$. (In other words, coordinates on $F_{Q_i}$ correspond to points of type $a+b_1+\ldots +b_{n-1}$ on $F_{Q}$,
where $a\in Q_i$ and $b_j$ are some fixed points on $Q_j$, $j\neq i$.)
In these coordinates the function $f|_Q$ splits into $n$ functions $f_{Q_1}, \ldots, f_{Q_n}$, such that $f_{Q_i}=f_{Q_j}$ for $Q_i=Q_j$.
This gives the required splitting $R_{Q,f}=B_1\cup\ldots\cup B_n$, where $B_i$ is isomorphic to $R_{Q_i,f_i}$ for $k_i>0$ and a standard linear section
$f_i$ on $\PP(1,1,a_i)$ as above, $B_j$ is a line (fiber)
for $k_j=0$, and $B_r=B_s$ for $Q_r=Q_s$.
\end{proof}

\begin{proposition}
\label{proposition:CY compactification}
Let $W$ be a smooth threefold. Let $F$ be a one-dimensional anticanonical linear system on $W$
with reduced fiber $D=F_\infty$.
Let a base locus $B\subset D$ be a union of smooth curves (possibly with multiplicities) such that for any two components
$D_1,D_2$ of $D$ one has $D_1\cap D_2\not\subset B$. Then there is a resolution of the base locus $f\colon Z\to \PP^1$
with a smooth total space $Z$ such that
$-K_Z=f^{-1}(\infty)$.
\end{proposition}

\begin{proof}[Proof (cf. Compactification construction~\ref{compactification construction}).]
Let $\pi\colon W'\to W$ be a blow up of one component $C$ of $B$ on $W$.
Since $\pi$ is a blow up of a smooth curve on a smooth variety, $W'$ is smooth.
Let $E$ be an exceptional divisor of the blow up.
Let $D'=\cup D'_i$ be a proper transform of $D=\cup D_i$.
Since multiplicity of $C$ in $D$ is $1$, one gets
$$
-K_{W'}=\pi^*(-K_{W})-E=D'+E-E=D'.
$$
Moreover, a base locus of the family on $W'$ is the same as $B$ or $B\setminus C$,
possibly together with a smooth curve $C'$ which is isomorphic to
$E\cap D'_i$;
in particular, $C$ is isomorphic to $\PP^1$.
(There are no isolated base points as the base locus is an intersection of two divisors on a smooth variety.)
Thus all conditions of the proposition hold for $W'$. Since $(W,F)$ is a canonical pair,
the base locus $B$ can be resolved in finite number of blow ups. This gives the required resolution.
\end{proof}

\begin{theorem}
\label{theorem:main}
Any Minkowski Laurent polynomial in three variables admits a log Calabi--Yau compactification.
\end{theorem}

\begin{proof}
Let $f$ be the Minkowski Laurent polynomial.
Recall that the Newton polytope $\Delta$ of $f$ is reflexive,
and the (singular Fano) toric variety whose fan polytope is $\nabla=\Delta^\vee$ is denoted by $T^\vee$.
The family of fibers of the map given by $f$ is a family $\{f=\lambda\}$, $\lambda\in \CC$.
Members of this family have natural compactifications to anticanonical sections of $T^\vee$. This family
(more precise, its compactification to a family $\{\lambda_0f=\lambda_1\}$ over $\PP[\lambda_0:\lambda_1]$)
is generated by a general member and the member that corresponds to the constant Laurent polynomial. The latter is nothing but
the boundary divisor $D$ of $T^\vee$. Denote the obtained pencil on $T^\vee$ by $f\colon Z_{T^\vee}\dashrightarrow \PP^1$
(we use the same notation $f$ for the Laurent polynomial, the corresponding pencil, and resolutions of this pencil for simplicity). By Lemma~\ref{lemma:restrictions Minkowski},
the base locus of $f$ on $Z_{T^\vee}$
is a union of smooth (rational) curves (possibly with multiplicities).
By Lemma~\ref{lemma: 3dim smoothness}, the variety $\widetilde T^\vee$ is a crepant resolution of $T^\vee$.
By definition of a Newton polytope, coefficients of the Minkowski Laurent polynomial at vertices of $\Delta$ are non-zero.
This means that the base locus does not contain any torus invariant strata of $T^\vee$ since it does not contain torus invariant
points by Fact $4$. 
Thus we get a family $f\colon Z_{\widetilde T^\vee}\dashrightarrow \PP^1$, whose total space is smooth and a base locus is
a union of (transversally intersecting) smooth curves (possibly with multiplicities) again.
By Proposition~\ref{proposition:CY compactification}, there is a resolution $f\colon Z\to \PP^1$ of the base locus on
$Z_{\widetilde T^\vee}$ such that $Z$ is smooth and $-K_Z=f^{-1}(\infty)$.
Thus $Z$ is the required log Calabi--Yau compactification, and
$Y=Z\setminus f^{-1}(\infty)$ is a Calabi--Yau compactification.
\end{proof}

\begin{remark}
The construction of Calabi--Yau compactification is not canonical: it depends on an order of blow ups of base locus components.
However all log Calabi--Yau compactifications are isomorphic in codimension one.
\end{remark}

There are $105$ families of smooth Fano threefolds. By Remark~\ref{remark: cDV threefolds} and~\cite{CCGK16}, there are $98$ ones among them that have degenerations to
toric varieties whose fan polytopes coincide with Newton polytopes of Minkowski Laurent polynomials satisfying period condition (for trivial divisors). Thus they have toric Landau--Ginzburg models by Theorem~\ref{theorem:main}.
Two other varieties, $X_{1-1}$ and $X_{1-11}$, have toric Landau--Ginzburg models by~\cite{Prz13}.
The following proposition, in a spirit of~\cite{Prz13}, proves the existence of (log) Calabi--Yau compactifications
for some degenerations of other threefolds.

\begin{proposition}
\label{proposition:CY hyperelliptic}
Fano threefolds $X_{2-1}$, $X_{2-2}$, $X_{2-3}$, $X_{9-1}$, and $X_{10-1}$
have toric Landau--Ginzburg models.
\end{proposition}

\begin{proof}
The Fano variety $X_{2-1}$ is a hypersurface section of type $(1,1)$ in $\PP^1\times X_{1-11}$ in an anticanonical embedding;
in other words, it is a complete intersection of hypersurfaces of types $(1,1)$ and $(0,6)$ in $\PP^1\times \PP(1,1,1,2,3)$.
The Fano variety $X_{2-2}$ is a hypersurface in a certain toric variety, see~\cite{CCGK16}.
The Fano variety $X_{2-3}$ is a hyperplane section of type $(1,1)$ in $\PP^1\times X_{1-12}$ in an anticanonical embedding;
in other words, it is a complete intersection of hypersurfaces of types $(1,1)$ and $(0,4)$ in $\PP^1\times \PP(1,1,1,1,2)$.
Finally one has $X_{9-1}=\PP^1\times S_2$ and $X_{10-1}=\PP^1\times S_1$.

For a variety $X_{i-j}$ construct its Givental's type Landau--Ginzburg models, see~\cite{Gi97}, for a compact
explanation see~\cite{PSh14}. Then present it by Laurent polynomial $f_{i-j}$, see, for instance,~\cite{Prz10}, and~\cite{Prz13}.
It satisfies the period condition by~\cite{fanosearch}, and it satisfies the toric condition by~\cite{IKKPS} and~\cite{DHKLP}.
In a spirit of~\cite{Prz13} compactify the family given by $f_{i-j}$ to a family of (singular)
anticanonical hypersurfaces in $\PP^1\times \PP^2$ or $\PP^3$
and then crepantly resolve singularities of a total space of the family.
Consider these cases one by one.

Givental's Landau--Ginzburg model for $X_{2-1}$ is
a complete intersection
$$
\left\{
  \begin{array}{ll}
    u+v_0=0,  \\
    v_1+v_2+v_3=0
  \end{array}
\right.
$$
in $\Spec \TT[u,v_0,v_1,v_2,v_3]$
with a function
$$
u+\frac{1}{u}+v_0+v_1+v_2+v_3+\frac{1}{v_1v_2^2v_3^3}.
$$
After birational change of variables (see~\cite{Prz10})
$$
v_1=\frac{x}{x+y+1},\ \ v_2=\frac{y}{x+y+1},\ \ v_3=\frac{1}{x+y+1},\ \ u=\frac{z}{z+1},\ \ v_0=\frac{1}{z+1}
$$
one, up to an additive shift, gets a function
$$
f_{2-1}=\frac{(x+y+1)^6(z+1)}{xy^2}+\frac{1}{z}
$$
on a torus $\Spec \TT[x,y,z]$.

Consider a family $\{f_{2-1}=\lambda\}$, $\lambda\in \CC$.
Make a birational change of variables (cf. the proof of~\cite[Theorem $18$]{Prz13})
$$
x=\frac{1}{b_1}-\frac{1}{b_1^2b_2}-1,\ \ y=\frac{1}{b_1^2b_2},\ \ z=\frac{1}{a_1}-1
$$
and multiply the obtained expression by a denominator. We see that the family is birational to
$$
\{(1-a_1)b_2^3=\left((1-a_1)\lambda-a_1\right)a_1(b_1b_2-b_1^2b_2-1)\}\subset \Aff[a_1,b_1,b_2]\times \Aff[\lambda].
$$
Now this family can be compactified to a family of hypersurfaces of bidegree $(2,3)$ in $\PP^1\times \PP^2$ using the embedding
$\TT[a_1,b_1,b_2]\hookrightarrow \PP[a_0:a_1]\times \PP[b_0:b_1:b_2]$.
The (non-compact) total space of the family has trivial canonical class and its singularities are a union of (possibly) ordinary double points and a rational
curves which are du Val along a line in general points. Blow up any of these curves. We get
singularities of the similar type again. After several crepant blow ups one approaches to a threefold with just ordinary double points;
these points admit algebraic small resolution. This resolution completes the construction of the Calabi--Yau compactification.
Note that the total space $(\CC^*)^3$ of the initial family is embedded to the resolution.

In the similar way one gets Calabi--Yau compactifications for the other varieties.
One has
$$
f_{2-2}=\frac{(x+y+z+1)^2}{x}+\frac{(x+y+z+1)^4}{yz}.
$$
A change of variables
$$
x=ab,\ \ y=bc,\ \ z=c-ab-bc-1
$$
applied to a family $\{f_{2-2}=\lambda\}$ and a multiplication on a denominator give a family of quartics
$$
ac^3=(c-ab-bc-1)(\lambda ab-c^2).
$$
The embedding $\Spec \TT[a,b,c]\hookrightarrow \PP[a:b:c:d]$ gives a compactification
to a family of singular quartics over $\Aff^1$.

One has 
$$
f_{2-3}=\frac{(x+y+1)^4(z+1)}{xyz}+z+1.
$$
A change of variables
$$
x=ac,\ \ y=a-ac-1,\ \ z=\frac{b}{c}-1
$$
applied to a family $\{f_{2-3}=\lambda\}$ and a multiplication on a denominator give a family
$$
a^3b=(\lambda c-b)(b-c)(a-ac-1).
$$
The embedding $\Spec \TT[a,b,c]\hookrightarrow \PP[a:b:c:d]$ gives a compactification
to a family of singular quartics over $\Aff^1$.

One has
$$
f_{9-1}=x+\frac{1}{x}+\frac{(y+z+1)^4}{yz}.
$$
A change of variables
$$
x=\frac{c}{b},\ \ y=ac,\ \ z=a-ac-1
$$
applied to a family $\{f_{9-1}=\lambda\}$ and a multiplication on a denominator give a family
$$
a^3b=(\lambda bc-b^2-c^2)(a-ac-1).
$$
The embedding $\Spec \TT[a,b,c]\hookrightarrow \PP[a:b:c:d]$ gives a compactification
to a family of singular quartics over $\Aff^1$.

One has
$$
f_{10-1}=\frac{(x+y+1)^6}{xy^2}+z+\frac{1}{z}.
$$
A change of variables
$$
x=\frac{1}{b_1}-\frac{1}{b_1^2b_2}-1,\ \ y=\frac{1}{b_1^2b_2},\ \ z=a_1
$$
applied to a family $\{f_{10-1}=\lambda\}$ and a multiplication on a denominator give a family
$$
a_1b_2^3=(\lambda a_1-a_1^2-1)(b_1b_2-b_1^2b_2-1).
$$
The embedding $\Spec \TT[a_1,b_1,b_2]\hookrightarrow \PP[a_0:a_1]\times\PP[b_0:b_1:b_2]$ gives a compactification
to a family of singular hypersurfaces of bidegree $(2,3)$ in $\PP^1\times \PP^2$ over $\Aff^1$.

In all cases total spaces of the families have crepant resolutions.
\end{proof}

In some cases the Calabi--Yau compactification can be constructed in another way, using multipotential technique (see, for example,~\cite{KP12}) and elliptic fibrations.

\begin{proposition}[A.\,Harder]
\label{proposition:harder}
The polynomial $f_{10-1}$ satisfies the Calabi--Yau condition.
\end{proposition}

\begin{proof}
Consider a surface $B=\Aff[w]\times \PP[s_0:s_1]$. 
Compactify the family given by $f_{10-1}$ to a family of elliptic curves over $B$ so that the projection onto $\Aff^1$ gives
(the partial compactification of) the initial family.
This Weierstrass fibration can be given by the equation
$$
Y^2 = X^3 + f_2 X + f_3$$
with
$$
f_2 =-\frac{1}{3}  \left(2w^2s_1^2 - 3w s_0 s_1+s_0^2 + s_1^2\right)^4,
$$
$$
f_3 = \frac{2}{27} \left(2w^2s_1^2 - 3w s_0 s_1 - 864 w s_1^2 + s_0^2 + 864 s_0s_1 + s_1^2\right) \left(2 w^2s_1^2 - 3 w s_0 s_1 + s_0^2 + s_1^2\right)^5.
$$
This fibration is singular and has degeneracy locus over $B$ given by the equation
\begin{multline*}
s_1 \left(w s_1 - s_0\right) \left(2w^2 s_1^2 - 3w s_0 s_1 - 432 w s_1^2 + s_0^2 + 432 s_0 s_1 + s_1^2\right)\cdot\\
\cdot \left(2 w^2 s_1^2 - 3w s_0 s_1 + s_0^2 + s_1^2\right)^{10} = 0.
\end{multline*}
Each component of this singular locus is a smooth curve in $B$. The singularities in the total space of this fibration are in the fibers over the curve given by
$$
2w^2 s_1^2 - 3w s_0 s_1 + s_0^2 + s_1^2 = 0.
$$
Above this curve, we get a curve of du Val singularities of type $E_8$. These singularities can be resolved by blowing up $8$ times. This gives a smooth variety $Y$ which is relatively compact fibered over $\Aff^1$. To see that this resolution is actually a Calabi--Yau variety, one can use the canonical bundle formula (\cite[p. 132]{Mi83}). The equation is basically
$K_Y = g^*(K_{B} + L)$,
where $g$ is a map $Y\to B$, and $L$ is a divisor on the base of the fibration, which is in this case is the pullback from $\PP^1$ to $B$ of section of $\cO_{\PP^1}(2)$. Therefore, $K_Y$ is the pullback of the trivial bundle on $B$, hence is itself trivial. Thus $Y$ is a Calabi--Yau compactification of the family given by $f_{10-1}$.
\end{proof}

Summarizing~\cite{Prz13}, Theorem~\ref{theorem:main}, Proposition~\ref{proposition:CY hyperelliptic}, and forthcoming papers~\cite{DHKLP} and~\cite{DHKLP},
one gets the following assertion.

\begin{corollary}
\label{corollary: main}
A pair of a smooth Fano threefold $X$ and a trivial divisor on it has a toric Landau--Ginzburg model.
Moreover, if $-K_X$ is very ample, then any Minkowski Laurent polynomial satisfying the period condition for $(X,0)$
is a toric Landau--Ginzburg model.
\end{corollary}

\begin{remark}
Let us recall that $\widetilde T$ is a smooth toric variety with $F(\widetilde T)=\Delta$.
Let $f$ be a \emph{general} Laurent polynomial with $N(f)=\Delta$. The Laurent polynomial $f$ is a toric Landau--Ginzburg model for a pair $(\widetilde T,D)$, where $D$ is a general $\CC$-divisor on $\widetilde T$.
Indeed, the period condition for it is satisfied by~\cite{Gi97}.
Following the compactification procedure described in the proof of Theorem~\ref{theorem:main}, one can see that the base locus $B$ is a union of smooth transversally intersecting curves (not necessary rational). This means that in the same way as above the statement of Theorem~\ref{theorem:main} holds for $f$, so that $f$
satisfies the Calabi--Yau condition (cf.~\cite{Harder}). Finally the toric condition holds for $f$ tautologically. Thus $f$ is a toric Landau--Ginzburg model for $(\widetilde T,D)$.
\end{remark}

\begin{problem}
Prove this for smooth Fano threefolds and any divisor. A description of Laurent polynomials for all Fano threefolds and
any divisor is contained in~\cite{DHKLP}.
\end{problem}

\begin{question}
Is it true that the Calabi--Yau condition follows from the period and the toric ones?
If not, what conditions should be put on a Laurent polynomial to hold the implication?
\end{question}

Another advantage of the compactification procedure described in Theorem~\ref{theorem:main} is that
it enables one to describe
``fibers of compactified toric Landau--Ginzburg models over infinity''. These fibers play an important role, say,
for computing Landau--Ginzburg Hodge numbers, see~\cite{KKP14} and~\cite{LP16} for detailed studying of the del Pezzo case. We summarize these considerations in the following assertion.

\begin{corollary}[{cf.~\cite[Conjecture 2.3.13]{Harder}}]
\label{proposition: fibers over infinity}
Let $f$ be a Minkowski Laurent polynomial. 
Let $\widetilde T^\vee$ be a (smooth) maximally triangulated toric variety such that $F(\widetilde T^\vee)=N(f)$,
and let $D$ be a boundary divisor of $\widetilde T^\vee$.
There is a log Calabi--Yau compactification $f\colon Z\to \PP^1$ with
$-K_Z=f^{-1}(\infty)=D$. In particular, $D$ consists of $\frac{\left(-K_{T_{N(f)}}\right)^3}{2}+2$ components combinatorially given by
a triangulation of a sphere. (This means that vertices of the triangulation correspond to components of $D$, edges correspond to intersections
of the components, and triangles correspond to triple intersection points of the components.)
\end{corollary}

\begin{proof}
Let $v$ be a number of vertices in a triangulation of $\nabla$; in other words, $v$ is a number of integral points on the boundary of $\nabla$,
or, the same, the number of components of $D$. Let $e$ be a number of edges in the triangulation of $\nabla$, and let $f$ be a number
of triangles in the triangulation. As the triangulation is a triangulation of a sphere, one has $v-e+f=2$.
On the other hand one has $2e=3f$. This means that $v=\frac{f}{2}+2$. The assertion of the corollary follows from the fact that
both $\left(-K_{T_{N(f)}}\right)^3$ and $f$ are equal to a normalized volume of $\nabla$.
\end{proof}

\begin{remark}
  Let $g=\frac{\left(-K_{X}\right)^3}{2}+1$ be the genus of Fano threefold $X$; in particular, $D$ consists of $g+1$ components.
Then one has $g+1=\dim |-K_X|$.
\end{remark}

\begin{remark}
The description of fibers of Landau--Ginzburg models over infinity as boundary divisors fits well to Mirror Symmetry considerations
from the point of view of~\cite{CKP13} and~\cite{IKKPS}. In these papers Fano varieties and their Landau--Ginzburg models are connected,
via their toric degenerations, by elementary transformations called \emph{basic links}. From our point of view they are given
by elementary subtriangulations of a sphere of boundary divisors.
\end{remark}

General fibers of compactified toric Landau--Ginzburg models for Fano threefolds are smooth K3 surfaces.
However some of them can be singular and even reducible.
Our considerations give almost no information about them; however singular fibers of Landau--Ginzburg models
are of special interest --- they contain information about derived categories of singularities.
There is a lack of examples of computing these categories.
More computable invariant is a number of components of reducible fibers.

\begin{conjecture}[{\cite[Conjecture 1.1]{PSh15}, see also~\cite{GKR12}}]
\label{conjecture:components}
Let $X$ be a smooth Fano variety of dimension $n$. Let $f_X$ be its toric Landau--Ginzburg model corresponding to the trivial divisor on $X$.
Let $k_{f_X}$ be a number of all components of all reducible fibers (without multiplicities) of a Calabi--Yau compactification
for $f_X$ minus the number of reducible fibers. One has
$$
h^{1,n-1}(X)=k_{f_X}.
$$
\end{conjecture}

This conjecture is proven for Fano threefolds of rank one (see~\cite{Prz13}) and for complete intersections (see~\cite{PSh15}).

\begin{problem}
Prove Conjecture~\ref{conjecture:components} for all Fano threefolds.
\end{problem}

\begin{remark}
Most of Fano threefolds have ``simple'' toric degenerations, say, degenerations to toric varieties with cDV singularities
(combinatorially this means that their fan polytopes have, except for the origin, integral points only on edges).
In these particular cases one can keep track of the exceptional divisors appearing at the resolution procedure
described in Proposition~\ref{proposition:CY compactification} and Theorem~\ref{theorem:main}.
That is, one can compute multiplicities of the base curves (each multiplicity greater than $1$ gives exceptional
divisors in fibers) and a local behavior of their intersections. Then, in a way similar to~\cite[Resolution Procedure 4.4]{PSh15},
one can compute the required number of components.
\end{remark}

\end{document}